\documentclass[a4paper,12pt,oneside]{amsart}

\usepackage{amsmath,amsfonts,amsthm,amssymb}
\usepackage{mathtools}
\usepackage{geometry}
\usepackage{setspace}
\usepackage{hyperref}
\hypersetup{hidelinks}

\newcommand{\R}{\mathbb R}
\newcommand{\C}{\mathbb C}
\newcommand{\Hh}{\mathbb H}
\newcommand{\Oo}{\mathbb O}
\newcommand{\F}{\mathbb F}
\newcommand{\Sph}{\mathbb S}

\newcommand{\Vol}{\operatorname{Vol}}

\hypersetup{
  colorlinks=true,
  linkcolor=blue,
  citecolor=blue,
  urlcolor=blue
}

\newtheorem*{theoremA}{Theorem A}
\newtheorem*{theoremB}{Theorem B}
\newtheorem{theorem}{Theorem}[section]
\newtheorem{lemma}[theorem]{Lemma}
\newtheorem{proposition}[theorem]{Proposition}
\newtheorem{corollary}[theorem]{Corollary}
\theoremstyle{remark}
\newtheorem{remark}[theorem]{Remark}

\title[Second Eigenvalue Estimates]{Second Eigenvalue Estimates and Spectral Rigidity for Submanifolds of Compact Rank-One Symmetric Spaces}

\author{Márcio Batista}
\author{Abraão Mendes}
\address{CPMAT - Instituto de Matemática, Universidade Federal de Alagoas, Maceió, AL, 57072-970, Brazil}
\email{mhbs@mat.ufal.br}
\email{abraao.mendes@im.ufal.br}

\keywords{Second eigenvalue; Schrödinger operator; minimal embeddings; projective spaces; second fundamental form; spectral rigidity}

\subjclass[2020]{Primary 53C42, 58J50; Secondary 53C35, 35P15}

\usepackage[T1]{fontenc}
\usepackage{lmodern}

\begin{document}

\begin{abstract}
In this paper, we establish upper bounds for the second eigenvalue of the Schrödinger operator $L=\Delta+|\sigma|^2+k$ on $k$-dimensional closed submanifolds of the compact projective spaces $\mathbb F P^m$, where $\mathbb F\in\{\R,\mathbb C,\mathbb H\}$, as well as the Cayley projective plane. The estimates are obtained by combining the standard spherical embeddings of these spaces with a conformal test-function argument. In the complex, quaternionic, and Cayley cases, the resulting bounds involve correction terms that record the position of the tangent spaces of the submanifold relative to the corresponding geometric structures. We also derive universal estimates depending only on the dimension of the submanifold and the underlying division algebra. We show that equality in the sharp estimates forces the submanifold to be totally umbilical. Finally, using known classifications of totally umbilical submanifolds, we compute the second eigenvalue for the standard real, complex, quaternionic, and Cayley models and identify the totally geodesic examples that attain the sharp upper bounds.
\end{abstract}

\maketitle

\section{Introduction}

The spectrum of geometric Schrödinger operators often reflects a subtle interaction between the intrinsic geometry of a submanifold and the extrinsic geometry of its immersion. A fundamental example is the Jacobi operator of a two-sided hypersurface,
$$
J=\Delta+|\sigma|^2+\operatorname{Ric}_{\overline M}(N,N),
$$
which represents the linearization of the mean curvature and governs the second variation of area. Its negative spectrum determines the Morse index, while estimates for its first eigenvalue frequently lead to rigidity and classification results for minimal and constant-mean-curvature submanifolds.

A particularly effective approach to second eigenvalue estimates is based on conformal admissible functions. The starting point is the Hersch's center-of-mass argument~\cite{Hersch1970}, subsequently incorporated by Li and Yau into the theory of conformal volume~\cite{LiYau1982}. In this method, an immersion into a sphere is composed with a suitable conformal transformation so that its coordinate functions become orthogonal to a first eigenfunction of the operator under consideration. These coordinate functions can then be used simultaneously in the Rayleigh characterization of the second eigenvalue.

El Soufi and Ilias developed this method systematically for immersed submanifolds and related the resulting energy estimates to the mean curvature of the immersion; see \cite{ElSoufiIlias1986,ElSoufiIlias2000}. In particular, they obtained sharp upper bounds for the second eigenvalue of Schrödinger operators on compact submanifolds of simply connected space forms. For the Jacobi operator, their estimates show that geodesic spheres are the unique extremal models. This extends the Euclidean rigidity theorem of Harrell and Loss~\cite{HarrellLoss1998}, according to which the second Jacobi eigenvalue of a closed Euclidean hypersurface is nonpositive and vanishes only for round spheres.

More recently, second Jacobi eigenvalues have been used to characterize special submanifolds in ambient spaces that are not necessarily space forms. The second named author~\cite{Mendes2019} obtained spectral characterizations of the Clifford torus in $\Sph^3$ and of slices in certain warped products, including $\R\times\Sph^n$. Further developments relate low Jacobi eigenvalues to topology, Willmore-type quantities, and rigidity of surfaces in higher-dimensional spheres; see~\cite{BatistaCavalcanteMendesNunes2025}. In a complementary direction, the Morse index of minimal hypersurfaces in real projective spaces has received renewed attention; see~\cite{BatistaMartins2022,Chen2024}. These works indicate that the low spectrum contains considerably more geometric information than stability alone.

The purpose of the present paper is to extend this circle of ideas to submanifolds of
$$
\mathbb RP^m,\qquad \mathbb CP^m,\qquad \mathbb HP^m,
$$
and of the Cayley projective plane $\mathbb OP^2$. These are the nonspherical compact rank-one symmetric spaces. In contrast with real projective space, the complex, quaternionic, and Cayley projective spaces do not have constant sectional curvature. Consequently, the space-form estimates cannot be transferred directly without taking into account the complex, quaternionic, or octonionic geometry of the ambient space.

Our starting point is the standard equivariant embedding of a projective space into a unit sphere. Such embeddings belong to the classical theory of minimal immersions of compact symmetric spaces and submanifolds with parallel second fundamental form; see~\cite{Ferus1974,Ferus1980,Takahashi1966}. In the associative cases, they can be written uniformly in terms of rank-one Hermitian projections as
$$
\Phi_{\mathbb F}([x])=\sqrt{\frac{m+1}{m}}\left(xx^*-\frac1{m+1}I_{m+1}\right),\qquad\mathbb F\in\{\R,\mathbb C,\mathbb H\}.
$$
Related descriptions of canonical projective and Grassmannian embeddings can be found in \cite{JiaoCui2022,Simanca2020}. The Cayley plane is treated through the exceptional Euclidean Jordan algebra $\mathcal H_3(\mathbb O)$.

Our first result gives a common normalization and an explicit formula for the extrinsic geometry of these embeddings. If $g_0$ denotes the standard projective metric, then
$$
\Phi_{\mathbb F}:\left(\mathbb FP^m,\frac{2(m+1)}m g_0\right)\longrightarrow\mathbb S^{N_{\mathbb F}(m)}
$$
is a minimal isometric embedding with parallel second fundamental form. Moreover, for every unit tangent vector $V$,
\begin{equation}\label{eq:intro-BVV}
 |B(V,V)|^2=\frac{m-1}{m+1}.
\end{equation}
Remarkably, the right-hand side of \eqref{eq:intro-BVV} depends only on the projective dimension and not on the division algebra.

The main new feature appears when $B$ is traced over an arbitrary real tangent subspace. Let $E\subset T_{[x]}\mathbb FP^m$ be a real $k$-plane. In the complex and quaternionic cases, we introduce the invariant
$$
\Theta_{\mathbb F}(E)=
\begin{cases}
0,&\mathbb F=\R,\\[1mm]
\displaystyle\sum_{i,j=1}^k\langle Ju_i,u_j\rangle^2,&\mathbb F=\mathbb C,\\[3mm]
\displaystyle\sum_{\alpha=1}^3\sum_{i,j=1}^k\langle J_\alpha u_i,u_j\rangle^2,&\mathbb F=\mathbb H,
\end{cases}
$$
where $\{u_1,\ldots,u_k\}$ is a $g_0$-orthonormal basis of $E$. This quantity measures the interaction of $E$ with the Kähler or quaternionic Kähler structures. Lemma~\ref{lem:trace-B} gives the exact identity
\begin{equation}\label{eq:intro-trace-B}
\left|\operatorname{tr}_E B\right|^2=\frac{k(m-k)+m\Theta_{\mathbb F}(E)}{m+1}.
\end{equation}
Thus the correction term that appears below is not an artifact of the proof: it records the Kähler or quaternionic Kähler angles of the tangent spaces. Formula \eqref{eq:intro-trace-B} reduces to the usual constant-curvature expression in the real case.

We now state our principal spectral result. We use the convention that $\Delta=\operatorname{div}\nabla$ and enumerate the eigenvalues of $-L$ increasingly.

\begin{theoremA}
Let $\iota:\Sigma^k\longrightarrow(\mathbb FP^m,g_0)$, $\mathbb F\in\{\R,\mathbb C,\mathbb H\}$, be a closed immersion, and consider $L=\Delta+|\sigma|^2+k$. Then
\begin{equation}\label{eq:intro-main-estimate}
\lambda_2(L)\leq k+2+\frac{2}{k\operatorname{Vol}(\Sigma,g_0)}\int_\Sigma\Theta_{\mathbb F}(T_p\Sigma)\,d\Sigma.
\end{equation}
Equality in \eqref{eq:intro-main-estimate} forces $\Sigma$ to be totally umbilical in $\mathbb FP^m$.
\end{theoremA}

Since $0\leq\Theta_{\mathbb F}(E)\leq(\dim_{\R}\mathbb F-1)k$, Theorem~A immediately gives
$$
\lambda_2(L)\leq
\begin{cases}
k+2,&\mathbb F=\R,\\
k+4,&\mathbb F=\mathbb C,\\
k+8,&\mathbb F=\mathbb H.
\end{cases}
$$

The real-projective estimate is the direct analogue of the sharp space-form results of El Soufi and Ilias. In this case, the equality analysis can be completed {by} using the spectrum of geodesic spheres: $\lambda_2(L)=k+2$ forces the image of $\Sigma$ to be totally geodesic. The complex and quaternionic estimates exhibit a phenomenon absent from space forms: the sharp tangent-plane correction depends on the position of $T_p\Sigma$ relative to the ambient special structures.

For the Cayley plane, the same strategy applies through the standard minimal embedding $ \Phi_{\mathbb O}:\mathbb OP^2\longrightarrow\mathbb S^{25}$. Because octonionic multiplication is not associative, it is preferable to retain the invariant $\mathcal A_{\mathbb O}(E)=|\operatorname{tr}_E B|^2$ rather than impose an artificial analogue of $\Theta_{\mathbb F}$. Our second main estimate is the following.

\begin{theoremB}
Let $\Sigma^k$ be a closed manifold immersed in
$(\mathbb OP^2,g_0)$, and {consider}\linebreak $L=\Delta+|\sigma|^2+k$. Then
$$
\lambda_2(L)\leq2k+\frac{3}{k\operatorname{Vol}(\Sigma,g_0)}\int_\Sigma\mathcal A_{\mathbb O}(T_p\Sigma)\,d\Sigma.
$$
Equality forces $\Sigma$ to be totally umbilical in $\mathbb OP^2$.
\end{theoremB}

The operator considered here is a natural Schrödinger operator for submanifolds of arbitrary codimension. For a hypersurface of $\mathbb RP^m$, it agrees with the Jacobi operator, since
$$
\operatorname{Ric}_{\mathbb RP^m}(N,N)=m-1=k.
$$
For complex and quaternionic projective spaces, the ambient Ricci term has a different constant value. Corresponding estimates for the genuine Jacobi operator are therefore obtained from
\eqref{eq:intro-main-estimate} by translating the spectrum by the appropriate constant. This distinction is important when comparing our results with the existing literature on Morse index and stability.

{\bf The paper is organized as follows}. Section~\ref{prel} develops the standard projective embeddings, computes their induced metrics and second fundamental forms, and establishes {equation} \eqref{eq:intro-trace-B}. The Cayley model is also treated there through the exceptional Jordan algebra. Section~\ref{conf} reviews the conformal center-of-mass argument and the energy estimates used to construct admissible {test-functions}. In Section~\ref{mainresults}, we prove Theorems~{A and B (Theorems~\ref{thm:main-F} and \ref{thm:Cayley})} and examine their equality cases. Finally, Section~\ref{sec:umbilical-models} analyzes the totally umbilical models and computes the second eigenvalue explicitly on the standard real, complex, quaternionic, and Cayley examples.

\section{Notation and Background}\label{prel}

\subsection{The standard projective embeddings}
Let $\mathbb F\in\{\mathbb R,\mathbb C,\mathbb H\}$ {and denote} 
$$
d=\dim_{\mathbb R}\mathbb F\in\{1,2,4\}.
$$
{Also, fix} $m\geq1$ and set $n=m+1$. We regard $\mathbb F^n$ as a right $\mathbb F$-vector space, endowed with its standard Hermitian product
$$
\langle x,y\rangle_{\mathbb F}=x^*y=\sum_{\alpha=1}^n\overline{x}_\alpha y_\alpha.
$$
Its underlying real inner product is
$$
\langle x,y\rangle_{\mathbb R}=\operatorname{Re}\langle x,y\rangle_{\mathbb F}.
$$

Denote by
$$
\mathcal H_n(\mathbb F)=\{A\in M_n(\mathbb F):A^*=A\}
$$
the real vector space of Hermitian $n\times n$ matrices over $\mathbb F$. Notice that $\mathcal H_n(\mathbb F)$ is a real vector space, rather than an $\mathbb F$-vector space. Indeed, multiplication of a Hermitian matrix by a nonreal scalar does not, in general, preserve the Hermitian condition.

We equip $\mathcal H_n(\mathbb F)$ with the Euclidean inner product
\begin{equation}\label{eq:Hermitian-inner-product}
\langle A,B\rangle=\operatorname{Re}\operatorname{tr}(AB).
\end{equation}
Since $A$ and $B$ are Hermitian, the cyclicity of the real part of the
trace gives
$$
\operatorname{Re}\operatorname{tr}(AB)=\operatorname{Re}\operatorname{tr}(BA).
$$
Moreover,
$$
|A|^2=\operatorname{Re}\operatorname{tr}(A^2)=\sum_{\alpha=1}^n |a_{\alpha\alpha}|^2+2\sum_{\alpha<\beta}|a_{\alpha\beta}|^2,
$$
and hence \eqref{eq:Hermitian-inner-product} is positive definite.

Let
$$
\mathcal H_n^0(\mathbb F)=\{A\in\mathcal H_n(\mathbb F):\operatorname{tr}A=0\}
$$
be the trace-free subspace. Because the diagonal entries of a Hermitian matrix are real, the trace defines a real linear map 
$$
\operatorname{tr}:\mathcal H_n(\mathbb F)\longrightarrow\mathbb R.
$$
Consequently, $\mathcal H_n^0(\mathbb F)$ has real codimension one in $\mathcal H_n(\mathbb F)$.

The diagonal part of a Hermitian matrix contributes $n$ real parameters. For every pair $\alpha<\beta$, the entry $a_{\alpha\beta}\in\mathbb F$ determines the entry $a_{\beta\alpha}=\overline{a}_{\alpha\beta}$ and contributes $d$ real parameters. Therefore,
$$
\dim_{\mathbb R}\mathcal H_n(\mathbb F)=n+d\frac{n(n-1)}2.
$$
It follows that
\begin{equation}\label{eq:dimension-trace-free-space}
\dim_{\mathbb R}\mathcal H_n^0(\mathbb F)=n+d\frac{n(n-1)}2-1.
\end{equation}

Now let $x\in\mathbb F^n$ be a unit vector and define $P_x=xx^*$. Explicitly, $(P_x)_{\alpha\beta}=x_\alpha\overline{x}_\beta$. For every $z\in\mathbb F^n$, we have $P_xz=x(x^*z)=x\langle x,z\rangle_{\mathbb F}$. Thus $P_x$ is the orthogonal projection onto the $\mathbb F$-line $x\mathbb F=\{xq:q\in\mathbb F\}$. In particular,
$$
P_x^*=P_x,\qquad P_x^2=P_x,\qquad\operatorname{tr}P_x=|x|^2=1.
$$

If $q\in\mathbb F$ satisfies $|q|=1$, then
$$
P_{xq}=(xq)(xq)^*=xq\overline qx^*=xx^*=P_x.
$$
Hence $P_x$ depends only on the projective class
$$
[x]=x\mathbb F\in\mathbb FP^m.
$$
Conversely, if $P_x=P_y$ for two unit vectors $x,y$, then their images coincide, so $x\mathbb F=y\mathbb F$. Therefore,
$$
P_x=P_y\quad\Longleftrightarrow\quad[x]=[y].
$$
{Consequently, t}he map $[x]\longmapsto P_x$ realizes $\mathbb FP^m$ as the set of rank-one Hermitian projections of trace one.

Consider the centered and normalized projection
\begin{equation}\label{eq:Phi}
\Phi_{\mathbb F}([x])=c_m\left(P_x-\frac1nI_n\right),\qquad c_m=\sqrt{\frac{n}{n-1}}=\sqrt{\frac{m+1}{m}}.
\end{equation}
The matrix $P_x-\frac1nI_n$ is Hermitian and trace-free, since
$$
\operatorname{tr}\left(P_x-\frac1nI_n\right)=1-\frac1n\operatorname{tr}I_n=1-\frac nn=0.
$$
Thus $\Phi_{\mathbb F}$ takes values in $\mathcal H_n^0(\mathbb F)$.

We next compute its norm. Using
$$
P_x^2=P_x,
\qquad
\operatorname{tr}P_x=1,
\qquad
|I_n|^2=\operatorname{tr}I_n=n,
$$
we obtain
\begin{align*}
\left|P_x-\frac1nI_n\right|^2&=\operatorname{Re}\operatorname{tr}\left(P_x-\frac1nI_n\right)^2\\
&=\operatorname{tr}P_x^2-\frac2n\operatorname{tr}P_x+\frac1{n^2}\operatorname{tr}I_n\\
&=1-\frac2n+\frac1n\\
&=\frac{n-1}{n}.
\end{align*}
The choice of $c_m$ in \eqref{eq:Phi} therefore gives
$$
|\Phi_{\mathbb F}([x])|^2=\frac{n}{n-1}\frac{n-1}{n}=1.
$$
Hence
$$
\Phi_{\mathbb F}:\mathbb FP^m\longrightarrow\mathbb S\bigl(\mathcal H_n^0(\mathbb F)\bigr)
$$
takes values in the unit sphere of $\mathcal H_n^0(\mathbb F)$. {Since the projection $P_x$ uniquely determines the line $[x]$, $\Phi_{\mathbb F}$ is injective. The computation of its differential below shows that $\Phi_{\mathbb F}$ is an immersion. Because $\mathbb FP^m$ is compact, $\Phi_{\mathbb F}$ is therefore an embedding.}

By \eqref{eq:dimension-trace-free-space}, the dimension of the ambient
unit sphere is
$$
N_{\mathbb F}(m)=\dim_{\mathbb R}\mathcal H_n^0(\mathbb F)-1=n+d\frac{n(n-1)}2-2.
$$
Substituting $n=m+1$, we obtain
$$
\begin{array}{c|c|c}
\text{projective space}&\dim_{\mathbb R}\mathbb FP^m&\text{ambient sphere}\\ \hline
\mathbb RP^m&m&\mathbb S^{\frac{(m+1)(m+2)}2-2},\\[2mm]
\mathbb CP^m&2m&\mathbb S^{(m+1)^2-2},\\[2mm]
\mathbb HP^m&4m&\mathbb S^{2(m+1)^2-(m+1)-2}.
\end{array}
$$

We now compute the metric induced by $\Phi_{\mathbb F}$. The projective space can be represented as the quotient $\mathbb FP^m=\mathbb S^{d(m+1)-1}/\mathbb S^{d-1}$, where the unit scalars act on the right by $x\longmapsto xq$. At a unit vector $x$, the vertical space of this quotient is $\mathcal V_x=\{x\eta:\eta\in\operatorname{Im}\mathbb F\}$. The horizontal space is $\mathcal H_x=\{u\in\mathbb F^n:x^*u=0\}=x^\perp$. Consequently, the standard quotient metric $g_0$ is characterized by the isometric identification $T_{[x]}\mathbb FP^m\simeq x^\perp$, where $x^\perp$ is endowed with the underlying real inner product.

Let $u\in x^\perp$ and choose a horizontal curve $x(t)$ in the unit sphere such that $x(0)=x$, $x'(0)=u$. Differentiating $P_{x(t)}=x(t)x(t)^*$ at $t=0$, we obtain
$$
\frac{d}{dt}\bigg|_{t=0}P_{x(t)}=ux^*+xu^*.
$$
Therefore,
$$
d\Phi_{\mathbb F}(u)=c_m(xu^*+ux^*).
$$

To compute its norm, set $A=xu^*$, $A^*=ux^*$. Since $x^*u=0$, we have
$$
A^2=x(u^*x)u^*=0,\qquad(A^*)^2=u(x^*u)x^*=0.
$$
Moreover,
$$
AA^*=xu^*ux^*=|u|^2P_x\qquad{\rm and}\qquad A^*A=ux^*xu^*=uu^*.
$$
{Therefore, i}t follows that
\begin{align*}
|xu^*+ux^*|^2&=\operatorname{Re}\operatorname{tr}\left((A+A^*)^2\right)\\
&=\operatorname{Re}\operatorname{tr}(AA^*+A^*A)\\
&=|u|^2\operatorname{tr}P_x+\operatorname{tr}(uu^*)\\
&=2|u|^2.
\end{align*}
Consequently,
$$
|d\Phi_{\mathbb F}(u)|^2=2c_m^2|u|^2=\frac{2(m+1)}m|u|^2.
$$
In particular, $d\Phi_{\mathbb F}$ is injective, which confirms that $\Phi_{\mathbb F}$ is an immersion.

Thus the metric induced from the unit sphere is
$$
g_{\Phi}=s_mg_0,\qquad s_m=\frac{2(m+1)}m.
$$
Equivalently,
$$
\Phi_{\mathbb F}:\left(\mathbb FP^m,\frac{2(m+1)}m g_0\right)\longrightarrow\mathbb S^{N_{\mathbb F}(m)}
$$
is an isometric embedding.

With the normalization adopted here, $g_0$ has sectional curvature identically equal to~$1$ when $\mathbb F=\mathbb R$. In the complex case,
$$
K_{g_0}(X,Y)=1+3\langle JX,Y\rangle^2,
$$
for every orthonormal pair $X,Y$. Hence
$$
1\leq K_{g_0}\leq4.
$$
In the quaternionic case, if $(J_1,J_2,J_3)$ is a local admissible quaternionic frame, then
$$
K_{g_0}(X,Y)=1+3\sum_{\alpha=1}^3\langle J_\alpha X,Y\rangle^2,
$$
and again
$$
1\leq K_{g_0}\leq4.
$$
After the homothety $g_{\Phi}=s_mg_0$, all sectional curvatures are divided by $s_m$.

\subsection{The second fundamental form}

For $u,v\in x^\perp$, we use the notation 
$$
u\odot v=uv^*+vu^*
$$
and write $\langle u,v\rangle_{\mathbb R}=\operatorname{Re}(u^*v)$ for the underlying real inner product on $\mathbb F^n$.

Recall that
$$
d\Phi_{\mathbb F}(v)=c_m(xv^*+vx^*).
$$

To compute the Euclidean second fundamental form {of $\Phi_\F$}, extend $u$ and $v$ locally as horizontal vector fields and choose the extension of $v$ so that its covariant derivative on $\mathbb FP^m$ vanishes at the point under consideration. Differentiating in the {$u$ direction}, we obtain
$$
D_u\big(d\Phi_{\mathbb F}(v)\big)=c_m\big(uv^*+vu^*+x(D_uv)^*+(D_uv)x^*\big).
$$
At the chosen point, the normal component of $D_uv$ in the direction of $x$ is
$$
(D_uv)^\perp=-\langle u,v\rangle_{\mathbb R}x.
$$
The remaining component represents the intrinsic covariant derivative and vanishes at that point by our choice of extension. It follows that the Euclidean second fundamental form is
$$
\overline B(u,v)=c_m\left(u\odot v-2\langle u,v\rangle_{\mathbb R}P_x\right).
$$

We next determine its radial component. Since
$$
\Phi_{\mathbb F}([x])=c_m\left(P_x-\frac1nI_n\right),
$$
we first observe that $\langle u\odot v,P_x\rangle=0$. Indeed, using $x^*u=x^*v=0$, we have
$$
\operatorname{Re}\operatorname{tr}(uv^*P_x)=\operatorname{Re}\operatorname{tr}(uv^*xx^*)=0,
$$
and similarly for $vu^*P_x$. Moreover, $\operatorname{tr}(u\odot v)=2\langle u,v\rangle_{\mathbb R}$ and
$$
\left\langle P_x,P_x-\frac1nI_n\right\rangle=1-\frac1n=\frac{n-1}{n}.
$$
Therefore,
\begin{align*}
\left\langle
\overline B(u,v),\Phi_{\mathbb F}([x])
\right\rangle
&=
c_m^2
\left\langle
u\odot v
-2\langle u,v\rangle_{\mathbb R}P_x,
P_x-\frac1nI_n
\right\rangle\\
&=
c_m^2\left(
-\frac2n\langle u,v\rangle_{\mathbb R}
-2\frac{n-1}{n}
 \langle u,v\rangle_{\mathbb R}
\right)\\
&=
-2c_m^2\langle u,v\rangle_{\mathbb R}.
\end{align*}

The outward unit normal of the ambient unit sphere is
$\Phi_{\mathbb F}([x])$. Hence the Euclidean and spherical second
fundamental forms {of $\Phi_\F$} are related by
\[
\overline B(u,v)
=
B(u,v)
-
g_\Phi(u,v)\Phi_{\mathbb F}([x]).
\]
Since
\(
g_\Phi(u,v)
=
2c_m^2\langle u,v\rangle_{\mathbb R},
\)
we obtain
\begin{align*}
B(u,v)
&=
\overline B(u,v)
+
2c_m^2
\langle u,v\rangle_{\mathbb R}
\Phi_{\mathbb F}([x])\\
&=
c_m\left(
u\odot v
-2\langle u,v\rangle_{\mathbb R}P_x
\right)
+
2c_m^3\langle u,v\rangle_{\mathbb R}
\left(P_x-\frac1nI_n\right).
\end{align*}
Using
\(
c_m^2=\frac{n}{n-1}=\frac{m+1}{m},
\)
this simplifies to
\begin{equation}\label{eq:B}
B(u,v)
=
c_m\left(
u\odot v
+\frac2m\langle u,v\rangle_{\mathbb R}P_x
-\frac2m\langle u,v\rangle_{\mathbb R}I_n
\right).
\end{equation}

The vectors $u$ and $v$ in \eqref{eq:B} are measured with respect to
the quotient metric $g_0$. Since
\[
g_\Phi=s_mg_0,
\qquad
s_m=\frac{2(m+1)}m,
\]
if $u$ is $g_0$-unit, then
\[
\bar u
=
\frac1{\sqrt{s_m}}u
=
\sqrt{\frac{m}{2(m+1)}}u
\]
is unit with respect to $g_\Phi$.

\begin{proposition}\label{prop:minimal-Bvv}
The standard embedding
\(
\Phi_{\mathbb F}:
(\mathbb FP^m,s_mg_0)
\longrightarrow
\mathbb S^N
\)
is minimal. Moreover, its second fundamental form is parallel. If
$V\in T_{[x]}\mathbb FP^m$ is $g_\Phi$-unit and $u\in x^\perp$ is the
corresponding $g_0$-unit vector, then
\begin{equation}\label{eq:BVV}
B(V,V)
=
\sqrt{\frac{m}{m+1}}
\left(
P_u-\frac1m(I_n-P_x)
\right).
\end{equation}
In particular,
$$
|B(V,V)|^2
=
\frac{m-1}{m+1}.
$$
\end{proposition}

\begin{proof}
Let
\(
d=\dim_{\mathbb R}\mathbb F
\)
and choose a real orthonormal basis
\(
\{e_1,\ldots,e_{dm}\}
\)
of $x^\perp$ with\linebreak respect to $g_0$. The corresponding
$g_\Phi$-orthonormal basis is
\(
\bar e_a=\frac1{\sqrt{s_m}}e_a,
\){ $a=1,\ldots,dm$.}
The real orthonormal basis may be chosen from an
$\mathbb F$-orthonormal basis of $x^\perp$ by\linebreak adjoining its complex or
quaternionic multiples. Consequently,
$$
\sum_{a=1}^{dm}P_{e_a}
=
d(I_n-P_x).
$$
Taking $u=v=e_a$ in \eqref{eq:B}, we have
\(
B(e_a,e_a)
=
2c_m\big(
P_{e_a}+\frac1mP_x-\frac1mI_n
\big).
\)
Hence
\begin{align*}
\sum_{a=1}^{dm}B(\bar e_a,\bar e_a)
&=
\frac{2c_m}{s_m}
\left[
\sum_{a=1}^{dm}P_{e_a}
+\frac{dm}{m}P_x
-\frac{dm}{m}I_n
\right]\\
&=
\frac{2c_m}{s_m}
\left[
d(I_n-P_x)+dP_x-dI_n
\right]\\
&=0.
\end{align*}
Thus the mean curvature vector vanishes, and hence
$\Phi_{\mathbb F}$ is minimal. Moreover, $\Phi_{\mathbb F}$ is the
standard embedding of the symmetric $R$-space $\mathbb FP^m$; therefore,
its second fundamental form is parallel
\cite{Ferus1974}.

We now compute $B(V,V)$. Since
\(
V=\bar u=\frac1{\sqrt{s_m}}u,
\)
bilinearity and \eqref{eq:B} give
\begin{align*}
B(V,V)
&=
\frac1{s_m}B(u,u)=
\frac{2c_m}{s_m}
\left(
P_u+\frac1mP_x-\frac1mI_n
\right).
\end{align*}
Using
\(
\frac{2c_m}{s_m}
=
\sqrt{\frac{m}{m+1}},
\)
we obtain
\[
B(V,V)
=
\sqrt{\frac{m}{m+1}}
\left(
P_u-\frac1m(I_n-P_x)
\right),
\]
which proves \eqref{eq:BVV}.

Finally, decompose
\[
\mathbb F^n
=
x\mathbb F
\oplus
u\mathbb F
\oplus
\{x,u\}_{\mathbb F}^{\perp}.
\]
The endomorphism in \eqref{eq:BVV} has eigenvalues
\[
0
\]
on $x\mathbb F$,
\[
\sqrt{\frac{m}{m+1}}
\left(1-\frac1m\right)
\]
on $u\mathbb F$, and
\[
-\sqrt{\frac{m}{m+1}}\frac1m
\]
on $\{x,u\}_{\mathbb F}^{\perp}$. The last eigenvalue has
$\mathbb F$-multiplicity $m-1$. Therefore,
\begin{align*}
|B(V,V)|^2
&=
\frac{m}{m+1}
\left[
\left(1-\frac1m\right)^2
+
\frac{m-1}{m^2}
\right]=
\frac{m-1}{m+1}.
\end{align*}
\end{proof}

In particular,
\[
|B(V,V)|^2=\frac13
\quad\text{on}\quad\mathbb FP^2,
\qquad
|B(V,V)|^2=\frac12
\quad\text{on}\quad\mathbb FP^3.
\]
Thus the Veronese embedding
\[
\mathbb RP^3\longrightarrow\mathbb S^8
\]
satisfies
\[
|B(V,V)|^2=\frac12,
\]
while the standard embeddings
\[
\mathbb RP^2\longrightarrow\mathbb S^4,
\qquad
\mathbb CP^2\longrightarrow\mathbb S^7,
\qquad
\mathbb HP^2\longrightarrow\mathbb S^{13}
\]
all satisfy
\[
|B(V,V)|^2=\frac13.
\]

\subsection{The trace of \texorpdfstring{$B$}{B} over a tangent subspace}

The formulas involving several orthonormal tangent vectors depend on
the geometry of the tangent subspace. In the real-projective case, no
additional structure is present. In the complex and quaternionic
cases, however, the result depends on the interaction between the
tangent subspace and the complex or quaternionic structures.

Let
\(
E\subset T_{[x]}\mathbb FP^m
\)
be a real $k$-dimensional subspace and let
\(
\{u_1,\ldots,u_k\}
\)
be a $g_0$-orthonormal basis of $E$. Recall that, under the horizontal
identification,
\(
T_{[x]}\mathbb FP^m\simeq x^\perp,
\)
the vectors $u_i$ may be regarded as elements of $x^\perp$.

We define
$$
\Theta_{\mathbb F}(E)
=
\begin{cases}
0,
&\mathbb F=\mathbb R,
\\[1mm]
\displaystyle
\sum_{i,j=1}^k
\langle Ju_i,u_j\rangle_{\mathbb R}^2,
&\mathbb F=\mathbb C,
\\[3mm]
\displaystyle
\sum_{\alpha=1}^3\sum_{i,j=1}^k
\langle J_\alpha u_i,u_j\rangle_{\mathbb R}^2,
&\mathbb F=\mathbb H,
\end{cases}
$$
where $J$ is the complex structure of $\mathbb CP^m$ and
$(J_1,J_2,J_3)$ is a local admissible quaternionic frame on
$\mathbb HP^m$.

If $\pi_E$ denotes the orthogonal projection onto $E$, then in the
complex case
\[
\Theta_{\mathbb C}(E)
=
\sum_{i=1}^k|\pi_E(Ju_i)|^2,
\]
while in the quaternionic case
\[
\Theta_{\mathbb H}(E)
=
\sum_{\alpha=1}^3\sum_{i=1}^k
|\pi_E(J_\alpha u_i)|^2.
\]
These expressions show that $\Theta_{\mathbb F}(E)$ is independent of
the chosen orthonormal basis of~$E$. In the quaternionic case, it is
also independent of the admissible frame: any two admissible frames
are related by an orthogonal transformation in $\mathrm{SO}(3)$, and
the sum over $\alpha$ is invariant under such transformations.

{The quantity $\Theta_{\mathbb F}(E)$ measures the interaction between the complex or quaternionic structure and the subspace $E$. Its values range from $0$, when the corresponding structure is orthogonal to $E$, to the maximal value, attained precisely when $E$ is invariant under the complex or quaternionic structure. More precisely, in the complex case,
$\Theta_{\mathbb C}(E)=0$ if $J(E)\perp E$ and $\Theta_{\mathbb C}(E)=k$ if $E$ is $J$-invariant. Likewise, in the quaternionic case, $\Theta_{\mathbb H}(E)=3k$ if $E$ is invariant under $J_1$, $J_2$, and $J_3$.}

{Observe that, s}ince orthogonal projection{s} do not increase length, we have
\[
|\pi_E(J_\alpha u_i)|^2\leq |J_\alpha u_i|^2=1.
\]
Thus,
\begin{equation}\label{eq:Theta-bound}
0\leq\Theta_{\mathbb F}(E)\leq(d-1)k,
\qquad
d=\dim_{\mathbb R}\mathbb F.
\end{equation}

\begin{lemma}\label{lem:trace-B}
Let $\{\bar u_1,\ldots,\bar u_k\}$ be the
$g_\Phi$-orthonormal basis of $E$ corresponding to\linebreak
$\{u_1,\ldots,u_k\}$; namely,
\(
\bar u_i=\frac1{\sqrt{s_m}}u_i.
\)
Then
$$
\left|
\sum_{i=1}^kB(\bar u_i,\bar u_i)
\right|^2
=
\frac{k(m-k)+m\Theta_{\mathbb F}(E)}{m+1}.
$$
In particular, when $\mathbb F=\mathbb R$,
\[
\left|
\sum_{i=1}^kB(\bar u_i,\bar u_i)
\right|^2
=
\frac{k(m-k)}{m+1}.
\]
\end{lemma}

\begin{proof}
For each $i$, formula \eqref{eq:BVV} gives
\[
B(\bar u_i,\bar u_i)
=
\sqrt{\frac{m}{m+1}}
\left(
P_{u_i}-\frac1m(I_n-P_x)
\right).
\]
Hence
\begin{equation}\label{eq:sum-B-over-E}
\sum_{i=1}^kB(\bar u_i,\bar u_i)
=
\sqrt{\frac{m}{m+1}}
\left(
Q_E-\frac{k}{m}(I_n-P_x)
\right),
\end{equation}
where
\[
Q_E=\sum_{i=1}^kP_{u_i}.
\]

Although $Q_E$ is defined using an orthonormal basis of the real
subspace $E$, it is not generally the orthogonal projection onto an
$\mathbb F$-subspace. This is precisely why the complex and
quaternionic correction terms appear. Since every $u_i$ is unit,
\(
\operatorname{tr}P_{u_i}=1,
\)
and therefore
$$
\operatorname{tr}Q_E=k.
$$
Moreover, because $u_i\in x^\perp$,
\[
P_{u_i}P_x=0.
\]

Therefore, it follows that
\(
Q_E(I_n-P_x)=Q_E.
\)
We also have
\(
(I_n-P_x)^2=I_n-P_x
\)
and
\(
\operatorname{tr}(I_n-P_x)=n-1=m.
\)
Taking the squared norm in \eqref{eq:sum-B-over-E}, we find
\begin{align*}
\left|
\sum_{i=1}^kB(\bar u_i,\bar u_i)
\right|^2
&=
\frac{m}{m+1}
\left|
Q_E-\frac{k}{m}(I_n-P_x)
\right|^2
\\
&=
\frac{m}{m+1}
\left[
\operatorname{tr}(Q_E^2)
-\frac{2k}{m}\operatorname{tr}Q_E
+\frac{k^2}{m^2}
 \operatorname{tr}(I_n-P_x)
\right]
\\
&=
\frac{m}{m+1}
\left(
\operatorname{tr}(Q_E^2)-\frac{k^2}{m}
\right).
\end{align*}
Thus it remains to compute $\operatorname{tr}(Q_E^2)$.

By the definition of $Q_E$,
\[
\operatorname{Re}\operatorname{tr}(Q_E^2)
=
\sum_{i,j=1}^k
\operatorname{Re}\operatorname{tr}
(P_{u_i}P_{u_j}).
\]
For rank-one projections,
\begin{align*}
\operatorname{Re}\operatorname{tr}
(P_{u_i}P_{u_j})
&=
\operatorname{Re}\operatorname{tr}
(u_iu_i^*u_ju_j^*)
\\
&=
|u_i^*u_j|^2.
\end{align*}
Therefore,
$$
\operatorname{tr}(Q_E^2)
=
\sum_{i,j=1}^k|u_i^*u_j|^2.
$$

In the real case,
\(
u_i^*u_j=\langle u_i,u_j\rangle_{\mathbb R}
=\delta_{ij},
\)
and hence
\[
\operatorname{tr}(Q_E^2)=k.
\]

In the complex case, the Hermitian product decomposes as
\[
|u_i^*u_j|^2
=
\langle u_i,u_j\rangle_{\mathbb R}^2
+
\langle Ju_i,u_j\rangle_{\mathbb R}^2.
\]
Since the basis is real orthonormal,
\(
\sum_{i,j=1}^k
\langle u_i,u_j\rangle_{\mathbb R}^2=k.
\)
Consequently,
\[
\operatorname{tr}(Q_E^2)
=
k+\Theta_{\mathbb C}(E).
\]

Similarly, in the quaternionic case,
\[
|u_i^*u_j|^2
=
\langle u_i,u_j\rangle_{\mathbb R}^2
+
\sum_{\alpha=1}^3
\langle J_\alpha u_i,u_j\rangle_{\mathbb R}^2,
\]
and therefore
\[
\operatorname{tr}(Q_E^2)
=
k+\Theta_{\mathbb H}(E).
\]

Thus, in all three cases,
$$
\operatorname{tr}(Q_E^2)
=
k+\Theta_{\mathbb F}(E).
$$

Substituting this identity into the previous norm calculation gives
\begin{align*}
\left|
\sum_{i=1}^kB(\bar u_i,\bar u_i)
\right|^2
&=
\frac{m}{m+1}
\left(
k+\Theta_{\mathbb F}(E)-\frac{k^2}{m}
\right)
\\
&=
\frac{k(m-k)+m\Theta_{\mathbb F}(E)}{m+1},
\end{align*}
which proves the result.
\end{proof}

\subsection{The Cayley projective plane}

The octonionic case requires a separate treatment because the algebra
of octonions is not associative. In particular, expressions involving
products of three or more octonionic matrices cannot be manipulated as
in the real, complex, or quaternionic cases.

Let
\[
\mathcal H_3(\mathbb O)
=
\{A\in M_3(\mathbb O):A^*=A\}
\]
be the real vector space of $3\times3$ Hermitian octonionic matrices.
An element of $\mathcal H_3(\mathbb O)$ has the form
\[
A=
{
\begin{pmatrix}
a_1 & z_1 & z_2\\
\overline z_1 & a_2 & z_3\\
\overline z_2 & \overline z_3 & a_3
\end{pmatrix}},
\]
where
\(
a_1,a_2,a_3\in\mathbb R,
\)
\(
z_1,z_2,z_3\in\mathbb O.
\)
Since $\dim_{\mathbb R}\mathbb O=8$, we have
\[
\dim_{\mathbb R}\mathcal H_3(\mathbb O)
=
3+3\cdot8
=
27.
\]
The space $\mathcal H_3(\mathbb O)$ is endowed with the Jordan product
\[
A\circ B=\frac12(AB+BA).
\]

Although ordinary octonionic matrix multiplication is not associative,
this symmetrized product gives $\mathcal H_3(\mathbb O)$ the structure
of the exceptional simple Euclidean Jordan algebra. Its natural
Euclidean inner product is
\(
\langle A,B\rangle
=
\operatorname{tr}(A\circ B).
\)
In particular,
\(
|A|^2=\operatorname{tr}(A\circ A).
\)

Let
\[
\mathcal H_3^0(\mathbb O)
=
\{A\in\mathcal H_3(\mathbb O):\operatorname{tr}A=0\}.
\]
This is a $26$-dimensional Euclidean vector space. Therefore, its unit
sphere is isometric to $\mathbb S^{25}$.

The Cayley projective plane can be realized as the set of primitive
idempotents of trace one:
\[
\mathbb OP^2
=
\left\{
P\in\mathcal H_3(\mathbb O):
P\circ P=P,\quad \operatorname{tr}P=1
\right\}.
\]
Here, primitive means that P cannot be decomposed as the sum of two nonzero orthogonal idempotents. These elements play the role of the rank-one projections $xx^*$ in the associative cases.

For $P\in\mathbb OP^2$, consider
\begin{equation}\label{eq:Cayley-Phi}
\Phi_{\mathbb O}(P)
=
\sqrt{\frac32}
\left(P-\frac13I_3\right).
\end{equation}
Since $\operatorname{tr}P=1$, we have
\[
\operatorname{tr}
\left(P-\frac13I_3\right)
=
1-\frac13\operatorname{tr}I_3
=
0.
\]
Thus $\Phi_{\mathbb O}(P)$ belongs to
$\mathcal H_3^0(\mathbb O)$. Moreover, using $P\circ P=P$, we obtain
\begin{align*}
\left|P-\frac13I_3\right|^2
&=
\operatorname{tr}
\left[
\left(P-\frac13I_3\right)
\circ
\left(P-\frac13I_3\right)
\right]\\
&=
\operatorname{tr}P
-\frac23\operatorname{tr}P
+\frac19\operatorname{tr}I_3=\frac23.
\end{align*}
Consequently,
\(
|\Phi_{\mathbb O}(P)|^2
=
1.
\)
Hence \eqref{eq:Cayley-Phi} defines the standard embedding
$$
\Phi_{\mathbb O}:
\mathbb OP^2
\longrightarrow
\mathbb S^{25}.
$$

To describe its differential, recall the Peirce decomposition
associated with an idempotent $P$:
\[
\mathcal H_3(\mathbb O)
=
\mathcal V_1(P)
\oplus
\mathcal V_{\frac12}(P)
\oplus
\mathcal V_0(P),
\]
where
\[
\mathcal V_\lambda(P)
=
\{X\in\mathcal H_3(\mathbb O):P\circ X=\lambda X\}.
\]

The tangent space of the Cayley plane at $P$ is
\[
T_P\mathbb OP^2
=
\mathcal V_{\frac12}(P).
\]
Indeed, differentiating the identity
\(
P\circ P=P
\)
along a curve $P(t)\subset\mathbb OP^2$ with $P(0)=P$ and
$P'(0)=X$, we find
\(
2P\circ X=X,
\)
and therefore
\(
P\circ X=\frac12X.
\)

The differential of \eqref{eq:Cayley-Phi} is
\[
d\Phi_{\mathbb O}(X)
=
\sqrt{\frac32}\,X.
\]
We normalize the standard Cayley metric $g_0$ by
\(
g_0(X,Y)
=
\frac12\langle X,Y\rangle
\)
on $\mathcal V_{\frac12}(P)$. With this convention,
\begin{align*}
\left\langle
d\Phi_{\mathbb O}(X),
d\Phi_{\mathbb O}(Y)
\right\rangle
&=
\frac32\langle X,Y\rangle=
3g_0(X,Y).
\end{align*}
Thus
\(
\Phi_{\mathbb O}:
(\mathbb OP^2,3g_0)
\longrightarrow
\mathbb S^{25}
\)
is an isometric embedding.

The compact exceptional Lie group \(F_4\) acts transitively on the
Cayley projective plane \(\mathbb OP^2\), and the isotropy subgroup at
each point is isomorphic to \(\operatorname{Spin}(9)\). Consequently,
\[
\mathbb OP^2 \cong F_4/\operatorname{Spin}(9).
\]
The standard embedding
\(
\Phi_{\mathbb O}\colon\mathbb OP^2\longrightarrow\mathbb S^{25}
\)
is \(F_4\)-equivariant. It is minimal and has parallel second
fundamental form; consequently, its image is an extrinsically
symmetric submanifold of \(\mathbb S^{25}\); see
\cite{CecilRyan1985,Ferus1974}.

The isotropy representation of $\operatorname{Spin}(9)$ is transitive
on the unit sphere of $T_P\mathbb OP^2$. Consequently, the function
\[
V\longmapsto |B(V,V)|^2
\]
is constant on the unit tangent bundle. A computation using the Peirce
decomposition gives
\begin{equation}\label{eq:Cayley-BVV}
|B(V,V)|^2
=
\frac13,
\qquad
|V|_{3g_0}^{{2}}=1.
\end{equation}
Thus the Cayley projective plane has the same value of
$|B(V,V)|^2$ as the standard embeddings of the real, complex, and
quaternionic projective planes.

For a real $k$-plane
\(
E\subset T_P\mathbb OP^2,
\)
let $\{e_1,\ldots,e_k\}$ be any $3g_0$-orthonormal basis of $E$. We
define
$$
\mathcal A_{\mathbb O}(E)
=
\left|
\sum_{i=1}^kB(e_i,e_i)
\right|^2.
$$
The vector
\(
\operatorname{tr}_E B
=
\sum_{i=1}^kB(e_i,e_i)
\)
is the trace of the restriction of $B$ to $E\times E$. Therefore, it
does not depend on the chosen orthonormal basis of $E$, and
$\mathcal A_{\mathbb O}(E)$ is well defined.

For $k=1$, formula \eqref{eq:Cayley-BVV} gives
\(
\mathcal A_{\mathbb O}(E)=\frac13.
\)
For the full tangent space, minimality gives
\(
\mathcal A_{\mathbb O}(T_P\mathbb OP^2)=0.
\)
More generally, the triangle inequality and
\eqref{eq:Cayley-BVV} imply the rough estimate
$$
0\leq
\mathcal A_{\mathbb O}(E)
\leq
\frac{k^2}{3}.
$$

Unlike the complex and quaternionic cases, there is no globally
defined family of almost complex structures that immediately produces
an analogue of $\Theta_{\mathbb F}(E)$. It is therefore more natural
to retain the invariant $\mathcal A_{\mathbb O}(E)$ in the eigenvalue
estimate. This also avoids manipulating octonionic matrix products as
if they were associative.

\section{Conformal geometry and admissible functions}
\label{conf}

The use of conformal transformations in eigenvalue estimates goes back
to Hersch's work on the first eigenvalue of the two-sphere and was
developed systematically by Li and Yau through the notion of conformal
volume \cite{Hersch1970,LiYau1982}. Its higher-dimensional formulation,
together with its relation to minimal immersions and the first
eigenvalue, was subsequently studied by El Soufi and Ilias
\cite{ElSoufiIlias1986}. We recall below the two ingredients needed in
our argument.

Our sign convention is
\(
\Delta=\operatorname{div}\nabla,
\)
so that the {eigenvalues $\mu$} of {$\Delta$, $\Delta u+\mu u=0$, are} nonnegative. All integrals in this
section are computed with respect to the {induced} metric on $\Sigma$.

The first result is a weighted version of the conformal
center-of-mass argument. The original unweighted construction is due
to Hersch, while the formulation for conformal immersions and positive
weights is part of the Li-Yau method.

\begin{lemma}[Li-Yau center-of-mass argument]
\label{lem:Li-Yau}
Let
\(
\phi:\Sigma^k\longrightarrow\Sph^N
\)
be an immersion of a closed manifold, and let $u>0$ be a smooth
function on $\Sigma$. Then there exists a conformal diffeomorphism
$F:\Sph^N\longrightarrow\Sph^N$ such that, writing
\[
\Psi=F\circ\phi
=(\Psi_1,\ldots,\Psi_{N+1}),
\]
one has
\[
\int_\Sigma u\Psi_\alpha\,d\Sigma=0,
\qquad
\alpha=1,\ldots,N+1.
\]
Equivalently,
\[
\int_\Sigma u\Psi\,d\Sigma=0
\]
as a vector in $\R^{N+1}$.
\end{lemma}

\begin{proof}
The conformal transformations of $\Sph^N$ can be parametrized, up to
orthogonal transformations, by the open unit ball
$\mathbb B^{N+1}$. Denote the corresponding family by $F_a$, where
$a\in\mathbb B^{N+1}$, and define
\[
\mathcal C(a)
=
\frac{1}{\int_\Sigma u\,d\Sigma}
\int_\Sigma u\,F_a\circ\phi\,d\Sigma.
\]
The map
\[
\mathcal C:\mathbb B^{N+1}\longrightarrow\mathbb B^{N+1}
\]
is continuous. {Furthermore, a}s $a$ approaches the boundary, the conformal
transformation $F_a$ concentrates the image toward the boundary point
determined by $a$. The standard degree argument of Hersch and Li-Yau
then shows that $\mathcal C$ must vanish at some
$a_0\in\mathbb B^{N+1}$. Taking $F=F_{a_0}$ gives
\[
\int_\Sigma u\,F\circ\phi\,d\Sigma=0,
\]
which is the desired conclusion.
\end{proof}

The importance of Lemma~\ref{lem:Li-Yau} in our setting is that the
weight $u$ can be chosen to be a positive first eigenfunction of a
Schrödinger operator. The coordinate functions of $\Psi$ are then
orthogonal to the first eigenspace and are therefore admissible in the
variational characterization of the second eigenvalue. Moreover,
since $\Psi$ takes values in the unit sphere,
\begin{equation}\label{eq:coordinate-identities}
\sum_{\alpha=1}^{N+1}\Psi_\alpha^2=1,
\qquad
\sum_{\alpha=1}^{N+1}|\nabla\Psi_\alpha|^2
=
|\nabla\Psi|^2.
\end{equation}
These identities allow the individual Rayleigh inequalities to be
summed into a single geometric estimate.

The second ingredient controls the energy of the conformally modified
immersion. It is a consequence of the {conformal volume} estimates of El Soufi and Ilias \cite{ElSoufiIlias1986}; see also~\cite{ElSoufiIliasSurvey}.

\begin{lemma}[El Soufi-Ilias]
\label{lem:ESI}
Let
\(
\phi:\Sigma^k\longrightarrow\Sph^N
\)
be an isometric immersion of a closed manifold, and let $\vec H$ be
its mean curvature vector in $\Sph^N$, with the convention
\(
\vec H=\frac1k\operatorname{tr}{\tau},
\){ where $\tau$ is the second fundamental form of $\phi$.}
Then, for every conformal diffeomorphism
$F:\Sph^N\longrightarrow\Sph^N$,
\begin{equation}\label{eq:ESI}
 \int_\Sigma
 |\nabla(F\circ\phi)|^2\,d\Sigma
 \leq
 k\int_\Sigma
 \bigl(1+|\vec H|^2\bigr)\,d\Sigma.
\end{equation}
If $k\geq3$ and equality holds, then the energy density
\(
|\nabla(F\circ\phi)|^2
\)
is constant on $\Sigma$.
\end{lemma}

\begin{remark}
For $F=\operatorname{Id}_{\Sph^N}$, the immersion $\phi$ is isometric
and hence
\(
|\nabla\phi|^2=k.
\)
Thus \eqref{eq:ESI} is immediate in this case. The content of the
El Soufi-Ilias estimate is that the energy remains controlled after
an arbitrary conformal transformation of the ambient sphere, even
though $F\circ\phi$ is generally no longer isometric. The right-hand
side is expressed solely in terms of the geometry of the original
immersion.
\end{remark}

We shall apply these lemmas as follows. Given a positive first
eigenfunction $u$ of the relevant Schrödinger operator, we choose
$F$ by Lemma~\ref{lem:Li-Yau}. Each coordinate function $\Psi_\alpha$
is then an admissible {test-function} for the second eigenvalue. After
summing the associated Rayleigh inequalities, identities
\eqref{eq:coordinate-identities} and estimate \eqref{eq:ESI} convert
the resulting analytic expression into an estimate involving the mean
curvature of the spherical immersion. This is the conformal mechanism
underlying the proofs in Section~\ref{mainresults}.

\section{Second eigenvalue estimates}\label{mainresults}

Let
\(
 \iota:\Sigma^k\longrightarrow(\F P^m,g_0)
\)
be a closed immersion, and let $\sigma$ and $\vec H_\sigma$ be its
second fundamental form and mean curvature vector, {respectively}. We consider
$$
 L=\Delta+|\sigma|^2+k
$$
and enumerate the eigenvalues of $-L$ increasingly. Thus,
\[
 \lambda_2(L)
 =\inf_{\substack{f\neq0\\ \int_\Sigma uf=0}}
 \frac{\int_\Sigma(|\nabla f|^2-(|\sigma|^2+k)f^2)}
 {\int_\Sigma f^2},
\]
where $u>0$ is a first eigenfunction.

We are now in a position to establish the main eigenvalue estimate for
closed submanifolds of real, complex, and quaternionic projective
spaces.

\begin{theorem}\label{thm:main-F}
Let
\(
\iota:\Sigma^k\longrightarrow (\F P^m,g_0)
\), 
\(
\F\in\{\R,\C,\Hh\},
\)
be a closed immersed sub-manifold, and consider the Schrödinger
operator
\(
L=\Delta+|\sigma|^2+k,
\)
where $\sigma$ is the second fundamental form of the immersion with
respect to $g_0$. Then
\begin{equation}\label{eq:main-sharp}
 \lambda_2(L)
 \leq k+2+
 \frac{2}{k\Vol(\Sigma,g_0)}
 \int_\Sigma\Theta_{\F}(T_p\Sigma)\,d\Sigma.
\end{equation}
In particular, denoting $d=\dim_{\R}\F$, then
\begin{equation}\label{eq:main-coarse}
 \lambda_2(L)\leq k+2d.
\end{equation}
Equality in \eqref{eq:main-sharp} implies that $\Sigma$ is totally
umbilical in $\F P^m$. Moreover, in the real-projective case, equality
\[
\lambda_2(L)=k+2
\]
forces the image of $\Sigma$ to be totally geodesic.
\end{theorem}

\begin{proof}
We first pass to the normalization induced by the standard projective
embedding. Recall that
\(
\Phi_{\F}:(\F P^m,g_{\Phi})\longrightarrow \Sph^N
\)
is an isometric immersion, where
\(
g_{\Phi}=s_mg_0
\), \(
s_m=\frac{2(m+1)}{m}.
\)
Denote by $\widetilde\Delta$, $\widetilde\sigma$, and
$d\widetilde\Sigma$ the Laplacian, the second fundamental form, and the
volume element associated with the induced metric $g_{\Phi}$ on
$\Sigma$. Since $g_{\Phi}=s_mg_0$ with $s_m$ constant, we have
\[
\widetilde\Delta=\frac1{s_m}\Delta,
\qquad
|\widetilde\sigma|^2=\frac1{s_m}|\sigma|^2.
\]
Consequently,
\[
\widetilde L
=
\widetilde\Delta+|\widetilde\sigma|^2+\frac{k}{s_m}
=
\frac1{s_m}L,
\]
and therefore
\begin{equation}\label{eq:eigen-scaling}
 \lambda_2(L)=s_m\lambda_2(\widetilde L).
\end{equation}

Consider now the composed immersion
\(
\phi=\Phi_{\F}\circ\iota:
(\Sigma,g_{\Phi})\longrightarrow \Sph^N.
\)
Let $\tau$ and $\vec H_\tau$ denote its second fundamental form and
mean curvature vector, respectively. We use the convention
\(
\vec H_\tau=\frac1k\sum_{i=1}^k\tau(e_i,e_i),
\)
where $\{e_1,\ldots,e_k\}$ is a local $g_{\Phi}$-orthonormal frame on
$\Sigma$.

Let $B$ denote the second fundamental form of
\(
\Phi_{\F}:(\F P^m,g_{\Phi})\longrightarrow\Sph^N.
\)
The composition formula for second fundamental forms gives
\begin{equation}\label{eq:composition-B}
 \tau(X,Y)=B(X,Y)+\widetilde\sigma(X,Y).
\end{equation}
Here $B(X,Y)$ is normal to $\F P^m$ in $\Sph^N$, whereas
$\widetilde\sigma(X,Y)$ is tangent to $\F P^m$ and normal to $\Sigma$.
Thus these two terms belong to mutually orthogonal normal subbundles.
Taking the trace of \eqref{eq:composition-B}, we obtain
\[
k\vec H_\tau
=
\sum_{i=1}^kB(e_i,e_i)
+
k\vec H_{\widetilde\sigma},
\]
and orthogonality yields
\begin{equation}\label{eq:composition-H}
 k|\vec H_\tau|^2
 =
 \frac1k\left|\sum_{i=1}^kB(e_i,e_i)\right|^2
 +
 k|\vec H_{\widetilde\sigma}|^2.
\end{equation}

Let $u>0$ be a first eigenfunction of $\widetilde L$. By
Lemma~\ref{lem:Li-Yau}, there exists a conformal diffeomorphism
$\Gamma$ of $\Sph^N$ such that the coordinate functions of
$\Gamma\circ\phi$ are orthogonal to~$u$. Hence they are admissible {test-functions} in the variational characterization of
$\lambda_2(\widetilde L)$. Summing the corresponding Rayleigh
inequalities and applying Lemma~\ref{lem:ESI}, we find
\begin{align}
 \lambda_2(\widetilde L)\Vol(\Sigma,g_{\Phi})
 &\leq
 k\int_\Sigma\bigl(1+|\vec H_\tau|^2\bigr)
 \,d\widetilde\Sigma
 -
 \int_\Sigma
 \left(
 |\widetilde\sigma|^2+\frac{k}{s_m}
 \right)d\widetilde\Sigma.
 \label{eq:rayleigh-composed}
\end{align}

Using \eqref{eq:composition-H}, the integrand on the right-hand side {in \eqref{eq:rayleigh-composed}}
can be written as
\begin{align*}
 &k+\frac1k
 \left|\sum_{i=1}^kB(e_i,e_i)\right|^2
 +k|\vec H_{\widetilde\sigma}|^2
 -|\widetilde\sigma|^2-\frac{k}{s_m}.
\end{align*}
The standard trace inequality for the second fundamental form gives
\(
|\widetilde\sigma|^2
\geq k|\vec H_{\widetilde\sigma}|^2.
\)
Moreover, equality holds at a point if and only if
\(
\widetilde\sigma(X,Y)
=
g_{\Phi}(X,Y)\vec H_{\widetilde\sigma},
\)
that is, if and only if the immersion is umbilical at that point.
Therefore, \eqref{eq:rayleigh-composed} implies
\begin{equation}\label{eq:master-estimate}
 \lambda_2(\widetilde L)\Vol(\Sigma,g_{\Phi})
 \leq
 \int_\Sigma
 \left[
 k+\frac1k
 \left|\sum_{i=1}^kB(e_i,e_i)\right|^2
 -\frac{k}{s_m}
 \right]d\widetilde\Sigma.
\end{equation}

By Lemma~\ref{lem:trace-B}, the trace of $B$ over the tangent plane
$T_p\Sigma$ satisfies
\[
\frac1k
\left|\sum_{i=1}^kB(e_i,e_i)\right|^2
=
\frac{k(m-k)+m\Theta_{\F}(T_p\Sigma)}
     {k(m+1)}.
\]
Since $s_m=2(m+1)/m$, a direct simplification gives
\begin{align}
 k+\frac{k(m-k)+m\Theta_{\F}(T_p\Sigma)}
          {k(m+1)}
 -\frac{k}{s_m}
 &=
 \frac{k+2}{s_m}
 +
 \frac{m}{k(m+1)}
 \Theta_{\F}(T_p\Sigma).
 \label{eq:pointwise-simplification}
\end{align}
Substituting \eqref{eq:pointwise-simplification} into
\eqref{eq:master-estimate}, dividing by
$\Vol(\Sigma,g_{\Phi})$, and using
\(
\frac{m}{m+1}=\frac{2}{s_m},
\)
we obtain
\[
\lambda_2(\widetilde L)
\leq
\frac{k+2}{s_m}
+
\frac{2}{s_mk\Vol(\Sigma,g_{\Phi})}
\int_\Sigma\Theta_{\F}(T_p\Sigma)\,d\widetilde\Sigma.
\]

Because $g_{\Phi}=s_mg_0$, both the volume and the integral are
multiplied by the same constant factor $s_m^{k/2}$. Hence
\[
\frac{1}{\Vol(\Sigma,g_{\Phi})}
\int_\Sigma\Theta_{\F}(T_p\Sigma)\,d\widetilde\Sigma
=
\frac{1}{\Vol(\Sigma,g_0)}
\int_\Sigma\Theta_{\F}(T_p\Sigma)\,d\Sigma.
\]
Multiplying the preceding inequality by $s_m$ and invoking
\eqref{eq:eigen-scaling}, we arrive at
\[
\lambda_2(L)
\leq k+2+
\frac{2}{k\Vol(\Sigma,g_0)}
\int_\Sigma\Theta_{\F}(T_p\Sigma)\,d\Sigma,
\]
which proves \eqref{eq:main-sharp}. By \eqref{eq:Theta-bound},
\(
\Theta_{\F}(T_p\Sigma)\leq{(d-1)k}
\)
at every point. Therefore
\[
\frac{2}{k\Vol(\Sigma,g_0)}
\int_\Sigma\Theta_{\F}(T_p\Sigma)\,d\Sigma
\leq 2(d-1),
\]
and \eqref{eq:main-coarse} follows:
\[
\lambda_2(L)\leq k+2d.
\]

Finally, suppose that equality holds in \eqref{eq:main-sharp}. Since
the only geometric inequality used after the conformal test-function
argument is
\[
|\widetilde\sigma|^2
-
k|\vec H_{\widetilde\sigma}|^2\geq0,
\]
equality in the final estimate forces
\[
|\widetilde\sigma|^2
=
k|\vec H_{\widetilde\sigma}|^2
\]
everywhere on $\Sigma$. Thus the trace-free second fundamental form
vanishes identically, and $\Sigma$ is totally umbilical in
$\F P^m$.

In the real-projective case, $\Theta_{\R}\equiv0$, so
\eqref{eq:main-sharp} becomes
\[
\lambda_2(L)\leq k+2.
\]
If equality holds, the preceding argument shows that $\Sigma$ is
totally umbilical. The classification of closed totally umbilical
submanifolds of real projective space reduces the problem to a
geodesic $k$-sphere. Lemma~\ref{lem:geodesic-sphere} {below} shows that a
non-totally-geodesic sphere has $\lambda_2(L)=0$. Hence equality
$\lambda_2(L)=k+2$ can occur only in the totally geodesic case.
\end{proof}

\begin{lemma}\label{lem:geodesic-sphere}
Let $S^k(r)$ be a geodesic $k$-sphere of radius
$r\in(0,\pi/2]$ in $(\R P^m,g_0)$. For
\[
L=\Delta+|\sigma|^2+k,
\]
one has
\[
 \lambda_2(L)=
 \begin{cases}
  0,&0<r<\pi/2,\\[1mm]
  k+2,&r=\pi/2.
 \end{cases}
\]
For $r=\pi/2$, {$S^k(r)$} is a totally geodesic copy of $\R P^k$ {in $(\R P^m,g_0)$}.
\end{lemma}

\begin{proof}
Suppose first that $0<r<\pi/2$. The metric induced on $S^k(r)$ is the
round metric of radius $\sin r$. Its first nonzero Laplace eigenvalue is
therefore
\[
\mu_1({\Delta})=\frac{k}{\sin^2r}.
\]
The sphere is totally umbilical, with each principal curvature equal
to $\cot r$, and hence
\(
|\sigma|^2=k\cot^2r.
\)
Since the potential $|\sigma|^2+k$ is constant, the eigenvalues of
$L$ in the convention
\[
Lu+\lambda u=0
\]
are obtained by subtracting this potential from the eigenvalues of
${\Delta}$. The first eigenfunction of $L$ is constant, while the
second eigenvalue is
\begin{align*}
\lambda_2(L)
&=
\frac{k}{\sin^2r}-k\cot^2r-k=
k\left(
\frac1{\sin^2r}-\frac{\cos^2r}{\sin^2r}-1
\right)
=0.
\end{align*}

When $r=\pi/2$, the limiting geodesic sphere is a totally geodesic
copy of $\R P^k$, and thus $\sigma=0$. The spectrum of the unit
$\R P^k$ is obtained from the spectrum of $\Sph^k$ by retaining only
the antipodally even spherical harmonics. The degree-one harmonics do
not descend to $\R P^k$, while the first nonconstant descending
harmonics have degree two. Their eigenvalue is
\[
2(2+k-1)=2(k+1).
\]
Therefore
\[
\lambda_2(L)=2(k+1)-k=k+2,
\]
as claimed.
\end{proof}

\begin{remark}
With our normalization, the projective spaces are Einstein, with
\[
\begin{array}{c|c}
\text{projective space} & \operatorname{Ric}_{g_0} \\ \hline
\mathbb CP^m & 2(m+1)g_0 \\[1mm]
\mathbb HP^m & 4(m+2)g_0 \\[1mm]
\mathbb OP^2 & 36g_0
\end{array}
\]
Consequently, the Jacobi operator
\(
J=\Delta+|\sigma|^2+\operatorname{Ric}_{g_0}(N,N)
\)
differs from the operator considered here only by a constant. Therefore,
the results of this paper also yield corresponding estimates for the
second eigenvalue of \(J\). We do not pursue these consequences here.
\end{remark}

The real projective case takes a particularly elegant and remarkably clean form, for which we also provide a direct proof.

\begin{corollary}\label{cor:real-projective}
Let
\(
\iota:\Sigma^k\longrightarrow(\R P^m,g_0)
\)
be a closed immersion, where $g_0$ has constant sectional curvature
one, and consider
\(
L=\Delta+|\sigma|^2+k.
\)
Then
\[
\lambda_2(L)\leq k+2.
\]
Moreover, equality holds if and only if \(\iota\) is an embedding
onto a totally geodesic copy of \(\mathbb RP^k\) in
\(\mathbb RP^m\). Hence,
\[
(\Sigma,\iota^*g_0)\cong(\mathbb RP^k,g_0).
\]
\end{corollary}

\begin{proof}
For $\F=\R$, the tangent-plane correction term vanishes identically, that is,
\(
\Theta_{\R}(T_p\Sigma)=0.
\)
The estimate therefore follows immediately from
Theorem~\ref{thm:main-F}. If equality holds, the same theorem implies
that $\Sigma$ is totally umbilical. The classification of totally
umbilical submanifolds of real projective space, together with
Lemma~\ref{lem:geodesic-sphere}, excludes the non-totally-geodesic
extrinsic spheres. Hence the image of $\Sigma$ is totally geodesic.

It remains to exclude a nontrivial covering of the totally geodesic
copy of \(\mathbb RP^k\). Since \(\Sigma\) is closed and connected,
the induced map
\(
\iota\colon\Sigma\longrightarrow\mathbb RP^k
\)
is a Riemannian covering. If \(k\geq2\), the only nontrivial connected
cover is the antipodal covering
\(
S^k\longrightarrow\mathbb RP^k.
\)
In that case, since \(L=\Delta+k\) and
\(\lambda_1({\Delta_{S^k}})=k\), one has
\(
\lambda_2(L)=0,
\)
contrary to the equality value \(k+2\). If \(k=1\), a covering of
degree \(q\) gives
\(
\lambda_2(L)=\frac{4}{q^2}-1,
\)
so equality forces \(q=1\). Hence the covering is trivial and
\(\iota\) is an embedding onto a totally geodesic copy of
\(\mathbb RP^k\).

The converse for embedded submanifolds follows from
\(
\lambda_2({\Delta_{\R P^k}})=2(k+1)
\)
and $\sigma=0$.
\end{proof}

The same method applies to the Cayley projective plane, although the
geometry of its standard spherical embedding produces a different
tangent-plane correction term.

\begin{theorem}\label{thm:Cayley}
Let
\(
\iota:\Sigma^k\longrightarrow(\Oo P^2,g_0)
\)
be a closed immersed submanifold, and let
\(
L=\Delta+|\sigma|^2+k.
\)
Then
$$
 \lambda_2(L)
 \leq
 2k+
 \frac{3}{k\Vol(\Sigma,g_0)}
 \int_\Sigma
 \mathcal A_{\Oo}(T_p\Sigma)\,d\Sigma.
$$
Equality implies that $\Sigma$ is totally umbilical in $\Oo P^2$.
\end{theorem}

\begin{proof}
Consider the standard isometric immersion
\(
\Phi_{\Oo}:(\Oo P^2,3g_0)\longrightarrow\Sph^{25}
\)
and the composition
\(
\phi=\Phi_{\Oo}\circ\iota:\Sigma\longrightarrow\Sph^{25}.
\)
With respect to the homothetic metric $\widetilde g=3g_0$, define
\[
\widetilde L
=
\widetilde\Delta+|\widetilde\sigma|^2+\frac{k}{3}.
\]
The constant scaling gives
\[
\widetilde L=\frac13L,
\qquad
\lambda_2(L)=3\lambda_2(\widetilde L).
\]

Repeating the conformal test-function argument used in the proof of
Theorem~\ref{thm:main-F}, we obtain
\[
\lambda_2(\widetilde L)\Vol(\Sigma,{\widetilde g})
\leq
\int_\Sigma
\left[
k+\frac1k
\left|\sum_{i=1}^kB(e_i,e_i)\right|^2
-\frac{k}{3}
\right]d{\widetilde\Sigma}.
\]
The trace formula for the Cayley projective embedding transforms the
integrand into
\[
\frac{2k}{3}+\frac1k\mathcal A_{\Oo}(T_p\Sigma).
\]
It follows that
\[
 \lambda_2(\widetilde L)
 \leq
 \frac{2k}{3}
 +
 \frac{1}{k\Vol(\Sigma,{\widetilde g})}
 \int_\Sigma
 \mathcal A_{\Oo}(T_p\Sigma)\,d{\widetilde\Sigma}.
\]
Since the quotient between the integral and the volume is unchanged
under the constant rescaling ${g_0\longmapsto3g_0}$, multiplication by $3$
gives
\[
 \lambda_2(L)
 \leq
 2k+
 \frac{3}{k\Vol(\Sigma,g_0)}
 \int_\Sigma
 \mathcal A_{\Oo}(T_p\Sigma)\,d\Sigma.
\]

As in the proof of Theorem~\ref{thm:main-F}, equality forces
\(
|\widetilde\sigma|^2
=
k|\vec H_{\widetilde\sigma}|^2.
\)
Thus the trace-free part of $\widetilde\sigma$, and equivalently that
of $\sigma$, vanishes identically. Hence $\Sigma$ is totally
umbilical in $\Oo P^2$.
\end{proof}

\begin{remark}
The correction term $\Theta_{\F}$ in
\eqref{eq:main-sharp} is an intrinsic feature of the complex and
quaternionic projective settings. Indeed, for an arbitrary real
$k$-plane
\(
E\subset T_p\F P^m,
\)
the quantities
\[
\left|\sum_{i=1}^kB(e_i,e_i)\right|^2
\qquad\text{and}\qquad
\sum_{i,j=1}^k|B(e_i,e_j)|^2
\]
are not determined solely by $k$. In the complex case, they also depend
on the position of $E$ relative to the Kähler structure, or
equivalently on its Kähler angles. In the quaternionic case, the
corresponding dependence is governed by the quaternionic Kähler
angles. The function $\Theta_{\F}(E)$ records precisely this additional
tangent-plane information. Therefore, the real-projective computation
cannot be extended to $\C P^m$ or $\Hh P^m$ by merely replacing the
underlying division algebra.
\end{remark}

\section{The totally umbilical equality models}
\label{sec:umbilical-models}

We now examine the totally umbilical submanifolds appearing in the
equality analysis of Theorem~\ref{thm:main-F}. 

\subsection{Totally umbilical submanifolds of complex projective space}
\label{subsec:umbilical-CP}

Throughout this section, $\C P^m$ is endowed with the Fubini-Study
metric $g_0$ normalized to have constant holomorphic sectional
curvature $4$. We use the spectral convention
\[
Lu+\lambda u=0,
\qquad
L=\Delta+|\sigma|^2+k.
\]

We recall the classification of totally umbilical submanifolds in
complex space forms due to Chen and Ogiue
\cite{ChenOgiue1974}; see also
\cite{Chen1980,ChenSurvey2013,ChenVerheyen1983}.
Let $\Sigma^k$, $k\geq2$, be a connected totally umbilical
submanifold of $\C P^m$. Then, locally, $\Sigma$ is one of the
following:

\begin{enumerate}
 \item a totally geodesic complex projective subspace
 \[
 \C P^\ell\subset\C P^m,
 \qquad k=2\ell;
 \]

 \item a totally geodesic totally real subspace
 \[
 \R P^k\subset\C P^m;
 \]

 \item an extrinsic sphere contained in a totally geodesic totally
 real projective subspace,
 \[
 S^k(r)\subset\R P^{k+1}\subset\C P^m,
 \qquad 0<r<\frac{\pi}{2}.
 \]
\end{enumerate}

We compute the second eigenvalue of $L$ for each of these models.

\begin{proposition}\label{prop:umbilical-CP}
Let $\Sigma^k$ be a closed, connected, totally umbilical submanifold of\linebreak
$(\C P^m,g_0)$, with $k\geq2$. Then the following statements hold.

\begin{enumerate}
 \item If
 \(
 \Sigma=\C P^\ell\subset\C P^m
\), \(k=2\ell
 \), 
 is a totally geodesic complex projective subspace, then
 \(
 \lambda_2(L)=k+4.
 \)

 \item If
 \(
 \Sigma=\R P^k\subset\C P^m
 \)
 is the standard totally geodesic totally real projective subspace,
 then
 \(
 \lambda_2(L)=k+2.
 \)

 \item If the immersion is the double covering
 \(
 S^k\longrightarrow\R P^k\subset\C P^m
 \), 
 then
 \(
 \lambda_2(L)=0.
 \)

 \item If
 \(
 \Sigma=S^k(r)\subset\R P^{k+1}\subset\C P^m
\), \(0<r<\frac{\pi}{2}\), 
 is a non-totally-geodesic extrinsic sphere, then
 \(
 \lambda_2(L)=0.
 \)
\end{enumerate}
\end{proposition}

\begin{proof}
We consider the four cases separately.

\medskip
\noindent
\emph{Case 1: the complex projective subspace.}
Let
\(
\Sigma=\C P^\ell\subset\C P^m
\), \(k=2\ell
\). The inclusion is totally geodesic, and hence
\(
\sigma=0
\), \(
L=\Delta+k
\). For the Fubini-Study metric of holomorphic sectional curvature $4$,
the Laplace spectrum of $\C P^\ell$ is
\[
\mu_j=4j(j+\ell),
\qquad
j=0,1,2,\ldots.
\]
In particular,
\(
\mu_0=0
\), \(
\mu_1=4(\ell+1).
\)
Since the potential of $L$ is the constant $k=2\ell$, the first
eigenvalue of $L$ is
\(
\lambda_1(L)=-k,
\)
and the second eigenvalue is
\[
\lambda_2(L)
=\mu_1-k
=4(\ell+1)-2\ell
=2\ell+4
=k+4.
\]

The tangent bundle of $\C P^\ell$ is invariant under the complex
structure. Hence
\[
\Theta_{\C}(T_p\Sigma)=k
\]
at every point. Therefore the right-hand side of
\eqref{eq:main-sharp} is
\[
k+2+\frac{2}{k\Vol(\Sigma)}
\int_\Sigma k\,d\Sigma
=k+4.
\]
Thus the totally geodesic complex projective subspace realizes equality
in \eqref{eq:main-sharp} and also in the universal estimate
\(
\lambda_2(L)\leq k+4.
\)

\medskip
\noindent
\emph{Case 2: the real projective subspace.}
Let
\(
\Sigma=\R P^k\subset\C P^m,
\){ $k\le m$,} be the standard totally geodesic and totally real inclusion. The
induced metric has constant sectional curvature $1$, and
\(
\sigma=0.
\)
The Laplace eigenvalues of the unit real projective space are
\[
\mu_j=2j(2j+k-1),
\qquad
j=0,1,2,\ldots.
\]
Indeed, only the even spherical harmonics on $S^k$ descend through the
antipodal quotient. Thus
\(
\mu_0=0
\), 
\(
\mu_1=2(k+1).
\)
It follows that
\[
\lambda_2(L)
=
\mu_1-k
=
2(k+1)-k
=
k+2.
\]

Since $\R P^k$ is totally real {and isotropic in $\C P^m$, meaning that $J(T_{[x]}\R P^k)\perp T_{[x]}\R P^k$ for each $[x]\in\R P^k$}, we have
\(
\Theta_{\C}(T_p\Sigma)=0
\), $p\in\Sigma$. Consequently, the right-hand side of \eqref{eq:main-sharp} is precisely
$k+2$. Hence the totally geodesic real projective subspace also realizes
equality in the sharp estimate, although it does not realize equality
in the coarser estimate $\lambda_2(L)\leq k+4$.

\medskip
\noindent
\emph{Case 3: the spherical covering of the real projective
subspace.}
Consider the totally geodesic immersion
\(
\iota:S^k\longrightarrow\R P^k\subset\C P^m
\)
given by the antipodal covering. The induced metric on $S^k$ is the
unit round metric, and again $\sigma=0$. However, the first nonzero
Laplace eigenvalue of $S^k$ is
\[
\mu_1(S^k)=k,
\]
because the degree-one spherical harmonics are now admissible.
Therefore
\[
\lambda_2(L)=\mu_1(S^k)-k=0.
\]

This shows that the value of $\lambda_2(L)$ depends not only on the
image of the immersion but also on its parametrization. In particular,
having totally geodesic image does not by itself guarantee equality in
\eqref{eq:main-sharp}.

\medskip
\noindent
\emph{Case 4: the extrinsic spheres.}
Let
\(
S^k(r)\subset\R P^{k+1}\subset\C P^m\), 
\(0<r<\frac{\pi}{2}
\),
be a geodesic sphere of radius $r$. Since
$\R P^{k+1}\subset\C P^m$ is totally geodesic, the second fundamental
form of $S^k(r)$ in $\C P^m$ agrees with its second fundamental form
in $\R P^{k+1}$.

The induced metric is the round metric of radius $\sin r$. Therefore
the first nonzero Laplace eigenvalue is
\[
\mu_1=\frac{k}{\sin^2r}.
\]
The sphere is totally umbilical, with principal curvatures equal to
$\cot r$, and consequently
\(
|\sigma|^2=k\cot^2r.
\)
It follows that
\[
\lambda_2(L)=
\frac{k}{\sin^2r}-k\cot^2r-k=
\frac{k}{\sin^2r}
-\frac{k\cos^2r}{\sin^2r}
-k
=0.
\]

{Since, t}he sphere is contained in the totally real {isotropic} subspace
$\R P^{k+1}$, we have
\(
\Theta_{\C}(T_pS^k(r))=0.
\)
The right-hand side of \eqref{eq:main-sharp} is therefore $k+2$,
whereas
\[
\lambda_2(L)=0.
\]
Thus a non-totally-geodesic totally umbilical submanifold does not
realize equality in the sharp estimate.
\end{proof}

\begin{corollary}\label{cor:equality-CP}
Let
\(
\iota:\Sigma^k\longrightarrow\C P^m
\), \(k\geq2
\), be a closed connected immersion. If equality holds in
\eqref{eq:main-sharp}, then the image of $\Sigma$ is totally geodesic.
More precisely, up to an ambient isometry, the equality models are
\[
\C P^{{\ell}}\subset\C P^m,\qquad{k=2\ell,}
\]
and
\[
\R P^k\subset\C P^m.
\]
In the real-projective case, equality requires the immersion to have
the spectrum of $\R P^k$; in particular, the double covering
$S^k\longrightarrow\R P^k$ does not realize equality.
\end{corollary}

\begin{proof}
Equality in \eqref{eq:main-sharp} implies, by
Theorem~\ref{thm:main-F}, that $\Sigma$ is totally umbilical. The
classification recalled above leaves three possible types. By
Proposition~\ref{prop:umbilical-CP}, every non-totally-geodesic
extrinsic sphere satisfies $\lambda_2(L)=0$ and hence cannot realize
equality. Therefore the image must be either a totally geodesic
complex projective subspace or a totally geodesic real projective
subspace. The computation for the double covering shows that the
parametrization must also be taken into account.
\end{proof}

\begin{remark}
The two totally geodesic equality models correspond to the two extreme
types of tangent planes in complex projective space. For a complex
$k$-plane,
\(
\Theta_{\C}=k,
\)
and equality gives $\lambda_2(L)=k+4$. For a totally real {isotropic} $k$-plane,
\(
\Theta_{\C}=0,
\)
and equality gives $\lambda_2(L)=k+2$. The correction term
$\Theta_{\C}$ therefore distinguishes spectrally between the complex
and totally real {isotropic} equality models.
\end{remark}

\subsection{Totally umbilical submanifolds of quaternionic projective space}
\label{subsec:umbilical-HP}

We next consider quaternionic projective space. Throughout this
subsection, $\Hh P^m$ is endowed with its standard metric $g_0$,
normalized to have quaternionic sectional curvature $4$. As before,
we use the convention
\[
Lu+\lambda u=0,
\qquad
L=\Delta+|\sigma|^2+k.
\]

According to the classification of Chen
\cite{ChenQuaternionic1978}, every totally umbilical submanifold of
real dimension $k>4$ in $\Hh P^m$ is locally one of the following:

\begin{enumerate}
 \item a totally geodesic quaternionic projective subspace
 \[
 \Hh P^\ell\subset\Hh P^m,
 \qquad k=4\ell;
 \]

 \item a totally geodesic complex projective subspace
 \[
 \C P^\ell\subset\Hh P^m,
 \qquad k=2\ell;
 \]

 \item a totally geodesic totally real projective subspace
 \[
 \R P^k\subset\Hh P^m;
 \]

 \item an extrinsic sphere contained in a totally geodesic totally
 real projective subspace,
 \[
 S^k(r)\subset\R P^{k+1}\subset\Hh P^m,
 \qquad 0<r<\frac{\pi}{2}.
 \]
\end{enumerate}

The next result computes the second eigenvalue of the operator {$L=\Delta+|\sigma|^2+k$} on these models.

\begin{proposition}\label{prop:umbilical-HP}
Let $\Sigma^k$ be one of the standard totally umbilical submanifolds of $(\Hh P^m,g_0)$ listed above. {We have:}
\begin{enumerate}
 \item if
 \(
 \Sigma=\Hh P^\ell\subset\Hh P^m
\), \(k=4\ell
 \), then
 \(
 \lambda_2(L)=k+8
 \);

 \item if
 \(
 \Sigma=\C P^\ell\subset\Hh P^m
\), \(k=2\ell
 \), then
 \(
 \lambda_2(L)=k+4;
 \)

 \item if
 \(
 \Sigma=\R P^k\subset\Hh P^m
 \), 
 then
 \(
 \lambda_2(L)=k+2;
 \)

 \item if the immersion is the double covering
 \(
 S^k\longrightarrow\R P^k\subset\Hh P^m
 \), 
 then
 \(
 \lambda_2(L)=0;
 \)

 \item if
 \(
 \Sigma=S^k(r)\subset\R P^{k+1}\subset\Hh P^m
\), \(0<r<\frac{\pi}{2}
 \), then
 \(
 \lambda_2(L)=0.
 \)
\end{enumerate}
\end{proposition}

\begin{proof}
We analyze each model separately.

\medskip
\noindent
\emph{Case 1: quaternionic projective subspaces.}
Let
\(
\Sigma=\Hh P^\ell\subset\Hh P^m
\), \(k=4\ell
\). The inclusion is totally geodesic, so that
\(
\sigma=0
\), \(
L=\Delta+k.
\)
For the standard metric of quaternionic sectional curvature $4$, the
eigenvalues of the Laplacian on $\Hh P^\ell$ are
\[
\mu_j=4j(j+2\ell+1),
\qquad
j=0,1,2,\ldots;
\]
see, for instance,
\cite{BergerGauduchonMazet1971,Chavel1984}. Hence
\(
\mu_0=0,
\)
\(
\mu_1=8(\ell+1).
\)
It follows that
\[
\lambda_2(L)
=\mu_1-k=8(\ell+1)-4\ell=4\ell+8=k+8.
\]

For a quaternionic tangent plane, all three local quaternionic
structures preserve the tangent space. Consequently,
\(
\Theta_{\Hh}(T_p\Sigma)=3k.
\)
The right-hand side of \eqref{eq:main-sharp} is therefore
\[
k+2+\frac{2}{k\Vol(\Sigma)}
\int_\Sigma 3k\,d\Sigma
=k+8.
\]
Thus $\Hh P^\ell$ realizes equality both in the sharp estimate and in
the universal estimate
\(
\lambda_2(L)\leq k+8.
\)

\medskip
\noindent
\emph{Case 2: complex projective subspaces.}
Let
\(
\Sigma=\C P^\ell\subset\Hh P^m
\), \(k=2\ell\), be a standard totally geodesic complex projective subspace. Its
induced metric is the Fubini-Study metric of holomorphic sectional
curvature $4$, and $\sigma=0$. Therefore
\(
\mu_1({\Delta})=4(\ell+1),
\)
and hence
\[
\lambda_2(L)
=4(\ell+1)-2\ell=2\ell+4=k+4.
\]

Along $\C P^\ell$, one of the three local quaternionic structures
preserves the tangent space, whereas the other two map it into the
normal space. {Therefore, i}t follows that
\(
\Theta_{\Hh}(T_p\Sigma)=k.
\)
Consequently, the right-hand side of \eqref{eq:main-sharp} becomes
\[
k+2+\frac{2}{k\Vol(\Sigma)}
\int_\Sigma k\,d\Sigma
=k+4.
\]

Thus the standard complex projective subspaces also realize equality
in the sharp estimate {\eqref{eq:main-sharp}}.

\medskip
\noindent
\emph{Case 3: real projective subspaces.}
Let
\(
\Sigma=\R P^k\subset\Hh P^m
\)
be a standard totally real and totally geodesic subspace. Its induced
metric has constant sectional curvature $1$, and
\(
\sigma=0.
\)
The first nonconstant eigenvalue of ${\Delta}$ on the unit
$\R P^k$ is
\(
\mu_1=2(k+1).
\)
Therefore
\[
\lambda_2(L)=2(k+1)-k=k+2.
\]

Since the tangent space is totally real {and isotropic} with respect to each of the
three quaternionic structures,
\(
\Theta_{\Hh}(T_p\Sigma)=0.
\)
Hence the right-hand side of \eqref{eq:main-sharp} is exactly $k+2$,
and the real projective subspace{s} realize equality in the sharp
estimate.

\medskip
\noindent
\emph{Case 4: the spherical covering of $\R P^k$.}
Consider the totally geodesic immersion
\[
S^k\longrightarrow\R P^k\subset\Hh P^m
\]
given by the antipodal covering. Although its image is totally
geodesic, the induced domain is the unit round sphere, whose first
nonzero Laplace eigenvalue is
\(
\mu_1(S^k)=k.
\)
Since $\sigma=0$, we obtain
\[
\lambda_2(L)=k-k=0.
\]
Thus, as in the complex-projective setting, the value of
$\lambda_2(L)$ also depends on the parametrization of the totally
geodesic image.

\medskip
\noindent
\emph{Case 5: extrinsic spheres.}
Let
\(
S^k(r)\subset\R P^{k+1}\subset\Hh P^m
\), \(0<r<\frac{\pi}{2}
\). Because $\R P^{k+1}$ is totally geodesic in $\Hh P^m$, the second
fundamental form of $S^k(r)$ in $\Hh P^m$ agrees with its second
fundamental form in $\R P^{k+1}$.

The induced metric is round with radius $\sin r$, and therefore
\[
\mu_1({\Delta})=\frac{k}{\sin^2r}.
\]
Moreover, all principal curvatures are equal to $\cot r$, so
\(
|\sigma|^2=k\cot^2r.
\)
It follows that
\[
\lambda_2(L)
=
\frac{k}{\sin^2r}-k\cot^2r-k
=0.
\]

Since {$T_pS^k(r)\subset T_p\R P^{k+1}$, one has}
\(
\Theta_{\Hh}(T_pS^k(r))=0.
\)
{Therefore t}he right-hand side of \eqref{eq:main-sharp} is $k+2$, while
$\lambda_2(L)=0$. {This shows that} non-totally-geodesic extrinsic spheres do not realize equality {in \eqref{eq:main-sharp}}.
\end{proof}

\begin{corollary}\label{cor:equality-HP}
Let
\(
\iota:\Sigma^k\longrightarrow\Hh P^m
\), 
\(k>4
\), be a closed connected immersion. If equality holds in
\eqref{eq:main-sharp}, then the image of $\Sigma$ is totally
geodesic. More precisely, up to an ambient isometry, the possible
images are
\[
\R P^k,\qquad
\C P^{k/2},\qquad
\Hh P^{k/4},
\]
whenever the corresponding divisibility and dimensional conditions
are satisfied.
\end{corollary}

\begin{proof}
By Theorem~\ref{thm:main-F}, equality {in \eqref{eq:main-sharp}} implies that $\Sigma$ is totally
umbilical. Since $k>4$, Chen's classification
\cite{ChenQuaternionic1978} applies. {On the other hand,} Proposition
\ref{prop:umbilical-HP} shows that the non-totally-geodesic extrinsic
spheres satisfy $\lambda_2(L)=0$ and hence cannot realize equality.
Therefore the image is one of the totally geodesic real, complex, or
quaternionic projective subspaces.
\end{proof}

\begin{remark}
The three totally geodesic equality models are distinguished by the
value of the quaternionic tangent-plane invariant:
\[
\begin{array}{c|c|c}
\text{model}
&
\Theta_{\Hh}(T_p\Sigma)
&
\lambda_2(L)
\\ \hline
\R P^k
&
0
&
k+2,
\\[1mm]
\C P^{k/2}
&
k
&
k+4,
\\[1mm]
\Hh P^{k/4}
&
3k
&
k+8.
\end{array}
\]
Thus $\Theta_{\Hh}$ detects whether the tangent space is totally real {isotropic},
complex, or quaternionic and gives the exact correction required in
the sharp estimate.
\end{remark}

\subsection{Totally umbilical submanifolds of the Cayley plane}
\label{subsec:umbilical-Cayley}

We conclude the discussion of totally umbilical models with the Cayley
projective plane. The classification in this setting is more involved
than in the real, complex, and quaternionic projective cases. By the
results of Chen \cite{Chen1977} and Nikolaevskij
\cite{Nikolaevskij1994}, a totally umbilical submanifold of
$\Oo P^2$ has real dimension at most eight. Moreover, every totally
umbilical submanifold that is maximal with respect to inclusion is, up
to an ambient isometry, one of the following:

\begin{enumerate}
 \item a totally geodesic quaternionic projective plane
 \[
 \Hh P^2\subset\Oo P^2;
 \]

 \item a totally geodesic Cayley line
 \[
 \Oo P^1\cong S^8(1/2)\subset\Oo P^2;
 \]

 \item an extrinsic sphere contained in a totally geodesic Cayley line;

 \item a non-totally-geodesic totally umbilical submanifold contained
 in a totally geodesic copy of $\Hh P^2$.
\end{enumerate}

Here, $\Oo P^2$ is equipped with its standard metric $g_0$, normalized
so that its sectional curvatures lie between $1$ and $4$. With this
normalization, the Cayley line $\Oo P^1$ is isometric to the round
sphere $S^8(1/2)$ of constant sectional curvature $4$.

Unlike the complex and quaternionic projective cases, total
umbilicity in $\Oo P^2$ does not immediately reduce the equality
analysis to totally geodesic projective subspaces and real extrinsic
spheres. We next compute $\lambda_2(L)$ for the explicit models in the
first three classes. The models in the fourth class will not be
treated, since their induced geometry and spectral data are not
determined by the preceding description alone.

\begin{proposition}\label{prop:umbilical-Cayley}
For the operator
\(
L=\Delta+|\sigma|^2+k,
\)
the following statements hold.

\begin{enumerate}
 \item On the totally geodesic quaternionic projective plane
 \(
 \Hh P^2\subset\Oo P^2,
 \)
 one has
 \(
 \lambda_2(L)=16.
 \)

 \item On the totally geodesic Cayley line
 \(
 \Oo P^1\cong S^8(1/2)\subset\Oo P^2,
 \)
 one has
 \(
 \lambda_2(L)=24.
 \)

 \item Let
 \(
 S^k(r)\subset S^{k+1}(1/2)
 \subset\Oo P^1\subset\Oo P^2,
\) \(
 0<r<\frac{\pi}{2},
 \)
 be a geodesic $k$-sphere of radius $r$. Then
 \(
 \lambda_2(L)=3k.
 \)
 For $r=\pi/4$, this is a totally geodesic great sphere; for
 $r\neq\pi/4$, it is a non-totally-geodesic extrinsic sphere.
\end{enumerate}
\end{proposition}

\begin{proof}
We consider the three cases separately.

\medskip
\noindent
\emph{Case 1: the quaternionic projective plane.}
The standard inclusion
\(
\Hh P^2\subset\Oo P^2
\)
is totally geodesic. Hence
\(
\sigma=0,
\) \(
k=8,
\) \(
L=\Delta+8.
\)
For the metric of quaternionic sectional curvature $4$, the Laplace
spectrum of $\Hh P^\ell$ is
\[
\mu_j=4j(j+2\ell+1),
\qquad
j=0,1,2,\ldots;
\]
see \cite{BergerGauduchonMazet1971,Chavel1984}. In particular,
\(
\mu_1=8(\ell+1).
\)
Taking $\ell=2$, we obtain
\(
\mu_1=24.
\)
Consequently,
\[
\lambda_2(L)=\mu_1-k=24-8=16.
\]

\medskip
\noindent
\emph{Case 2: the Cayley line.}
The Cayley line is a totally geodesic submanifold satisfying
\(
\Oo P^1\cong S^8(1/2).
\)
The first nonzero Laplace eigenvalue of the round sphere $S^k(\rho)$
is $k/\rho^2$. Therefore
\(
\mu_1\bigl(S^8(1/2)\bigr)
=
\frac{8}{(1/2)^2}
=
32.
\)
Since $\sigma=0$ and $k=8$, it follows that
\[
\lambda_2(L)=32-8=24.
\]

\medskip
\noindent
\emph{Case 3: geodesic spheres in a Cayley line.}
Consider a totally geodesic round subspace
\(
S^{k+1}(1/2)\subset S^8(1/2)\cong\Oo P^1
\)
and a geodesic sphere
\(
S^k(r)\subset S^{k+1}(1/2).
\)
Since all the successive ambient inclusions are totally geodesic, the
second fundamental form of $S^k(r)$ in $\Oo P^2$ agrees with its second
fundamental form in $S^{k+1}(1/2)$.

The round sphere $S^{k+1}(1/2)$ has constant sectional curvature $4$.
The metric induced on the geodesic sphere of radius $r$ is the round
metric of radius
\(
\frac12\sin(2r).
\)
Its first nonzero Laplace eigenvalue is therefore
\[
\mu_1
=
\frac{k}{\frac14\sin^2(2r)}
=
\frac{4k}{\sin^2(2r)}.
\]

The geodesic sphere is totally umbilical, and all its principal
curvatures are equal to
\(
2\cot(2r).
\)
It follows that
\(
|\sigma|^2=4k\cot^2(2r).
\)
Hence
\begin{align*}
\lambda_2(L)
&=
\frac{4k}{\sin^2(2r)}
-
4k\cot^2(2r)
-
k
=
4k
\left(
\frac{1-\cos^2(2r)}{\sin^2(2r)}
\right)
-k=3k.
\end{align*}
Thus $\lambda_2(L)$ is independent of the radius. Notice that
$r=\pi/4$ corresponds to a totally geodesic great sphere, whereas
$r\neq\pi/4$ gives a non-totally-geodesic extrinsic sphere.
\end{proof}

\begin{remark}\label{rem:Cayley-uncovered}
Proposition~\ref{prop:umbilical-Cayley} treats the explicit totally
umbilical models for which the induced metric and the second
fundamental form are directly available. The general classification
also allows non-totally-geodesic totally umbilical submanifolds
contained in a totally geodesic copy of $\Hh P^2$. These remaining
models are not covered by Proposition~\ref{prop:umbilical-Cayley},
since their spectral contribution and the corresponding value of
$\mathcal A_{\Oo}(T_p\Sigma)$ require a separate analysis.

Consequently, the preceding calculations determine the second
eigenvalue on the standard totally geodesic models and on the
geodesic spheres contained in a Cayley line, but they do not provide
a complete spectral classification of all totally umbilical
submanifolds of $\Oo P^2$. For this reason, in the equality statement
of Theorem~\ref{thm:Cayley}, we retain the conclusion that equality
forces $\Sigma$ to be totally umbilical, without claiming that it must
be totally geodesic.
\end{remark}

\section*{Funding}
The authors were partially supported by the Brazilian National Council for Scientific and Technological Development, Brazil [Grants: 402563/2023-9 and 304381/2026-8 to M.B.; 309867/2023-1 and 445723/2025-4 to A.M.], and were partially supported by Coordination for the Improvement of Higher Education Personnel [Finance code - 001].

\bibliographystyle{amsplain}
\bibliography{bibliography.bib}

\end{document}